\documentclass{amsproc}
\usepackage{breqn,float,amssymb,pdflscape,rotating,setspace,hyperref}
\makeatletter
\@namedef{subjclassname@2020}{%
  \textup{2020} Mathematics Subject Classification}
\makeatother
\newtheorem{theorem}{Theorem}[section]

\theoremstyle{definition}

\newtheorem{example}[theorem]{Example}

\theoremstyle{remark}

\numberwithin{equation}{section}
\allowdisplaybreaks
\begin{document}

\title{A short note on a definite integral of a general form}


\author{Robert Reynolds}
\address[Robert Reynolds]{Department of Mathematics and Statistics, York University, Toronto, ON, Canada, M3J1P3}
\email[Corresponding author]{milver73@gmail.com}
\thanks{}


\subjclass[2020]{Primary  30E20, 33-01, 33-03, 33-04}

\keywords{Definite integral, Cauchy integral theorem, Cauchy principal value, Catalan's constant, Ap\'{e}ry's constant, Malmsten}

\date{}

\dedicatory{}

\begin{abstract}
In this note we derive some interesting definite integrals involving Malmsten logarithm forms, reciprocal logarithm forms and K\"{o}lbig type integrals in terms of special functions.
\end{abstract}

\maketitle
\section{Introduction}
The famous book on Tables of integrals and series by Gradshteyn and Ryzhik \cite{grad}, currently edited by Dan Zwillinger [Gradshteyn and Ryzhik,\href{https://www.mathtable.com/gr/}{2014}] and Professor Victor Moll \cite{moll}, has many interesting definite integrals without derivations. In this note we look at one integral listed as (3.228.6(12)) in \cite{grad}, which we think is of interest to the scientific community. We produce the correct version of this integral, a formal derivation and evaluations in terms of fundamental constants and special functions. We evaluate nested $\log(\log(x))$  forms which were initially studied by Malmsten \cite{malmsten,malmsten1}, and further studied by Blagouchine \cite{blagouchine}, Moll \cite{moll1}, Median \cite{medina} and listed in Table 148 in \cite{bdh}. Further reading on these types of integrals are on a Wikipedia website [Malmsten,\href{https://en.wikipedia.org/wiki/Carl_Johan_Malmsten}{Malmsten}] hosting more information.\\\\
The derivation of the definite integral of a general form involving powers of the logarithm and other powers is achieved using contour integration \cite{reyn4}. Section (3.228) in Gradshteyn and Ryzhik contains integrals which have many applications in mathematics. The integral being studied in this work could be viewed as a Stieltjes transform see chapter XIV in \cite{erdt2}. The Stieltjes transform can be found in works on Laplace transforms by Widder (1941, Chapter VIII) \cite{widder} and Titchmarsh (1937, sections 11.8, 11.9) \cite{titchmarsh}. Stieltjes transforms are also connected with the moment problem for the semi-infinite interval (Shohat and Tamarkin, 1943) and hence with certain continued fractions \cite{shohat}.\\\\
Some applications of integrals in section (3.228) of  Gradshteyn and Ryzhik are; in obtaining quadrature formulae for singular integrals are derived that retain the nice features of Gauss Chebyshev quadrature \cite{venturino}, determining the support of the equilibrium measure in the presence of a monomial external field on [-1,1], \cite{damelin}, considering the tangential-displacement effects in the wedge indentation of an elastic half-space - an integral-equation approach \cite{georgiadis}, and in the study of dynamic frictional indentation of an elastic half-plane by a rigid punch \cite{brock}.\\\\
Errata: A possible derivation of Eq.  (3.228.6(12)) in \cite{grad} is a special case of Eq. (3.228.5(12)) in \cite{grad} and Mathematica by Wolfram given as;
\begin{multline}\label{eq1}
\int_0^{\infty } \frac{x^{v-1} (x+a)^{1-u}}{x-c} \, dx\\
=\frac{(-c)^{-u+v} (a+c) \left(\frac{a+c}{c}\right)^{-u} \pi  \csc (\pi  (u-v))}{c}\\+\frac{a^{1-u+v} \pi  \csc (\pi  (u-v))
   \Gamma (v) \, _2F_1\left(1,v;2-u+v;-\frac{a}{c}\right)}{c \Gamma (-1+u) \Gamma (2-u+v)}
\end{multline}
where $Re(a)>0,Re(c)\leq 0,Re(u)>Re(v)>0$.\\
When $u$ is an integer, $n$ in equation (\ref{eq1}), we get the correct form of Eq.  (3.228.6(12)) in \cite{grad} given by;
\begin{multline}\label{eq2}
\int_0^{\infty } \frac{x^{v-1} (\gamma +x)^{-n}}{x+b} \, dx=\frac{\pi  b^{v-1} }{\sin (\pi  v) (\gamma -b)^n}\left(1-\left(\frac{\gamma
   }{b}\right)^{v-1} \sum _{j=0}^{n-1} \frac{(1-v)_j \left(\frac{\gamma -b}{\gamma }\right)^j}{j!}\right)
\end{multline}
where $|\arg{\beta}|<\pi,|\arg{\gamma}|<\pi,0< Re(v) < n$.
\section{Preliminaries}
Throughout this work we will use the Hurwitz-Lerch zeta function given in [DLMF,\href{https://dlmf.nist.gov/25.14}{25.14}], and its special cases; the Hurwitz zeta function $\zeta(s,a)$, [DLMF,\href{https://dlmf.nist.gov/25.11}{25.11}] and the polylogarithm $Li_{s}(z)$, [DLMF,\href{https://dlmf.nist.gov/25.12#ii}{25.12(ii)}], the log-gamma function given in [DLMF,\href{https://dlmf.nist.gov/25.11#vi.info}{25.11.18}], Catalan's constant $C$, [DLMF,\href{https://dlmf.nist.gov/26.6.E12}{25.11.40}], Glaisher's constant $A$, [DLMF,\href{https://dlmf.nist.gov/5.17.E5}{5.17.5}], Apery's constant $\zeta(3)$, [Wolfram,\href{https://mathworld.wolfram.com/AperysConstant.html}{1}], the Pochhammer's symbol $(a)_{n}$, [DLMF,\href{https://dlmf.nist.gov/5.2.iii}{5.2.5}], Stieltjes constant $\gamma_{n}$, [DLMF,\href{https://dlmf.nist.gov/25.2.E5}{25.2.5}]
\section{Contour integral representations}
In this section we apply the method in \cite{reyn4} to equation (\ref{eq2}) to derive the contour integral representations used to derive the main theorem in this work, proceeded by its evaluations.  The interchange of multiple integrals and integral and summation are achieved by Tonelli's and Fubini's theorems see pp. 177 in \cite{gelca}. Based on equation (\ref{eq2}) the three contour integral representations are;
\subsection{Left-hand side contour integral}
\begin{multline}\label{eq:d1}
\int_{0}^{\infty}\frac{x^{-1+m} (x+\gamma )^{-n} \log ^k(a x)}{(b+x) \Gamma(k+1)}dx\\
=\frac{1}{2\pi i}\int_{0}^{\infty}\int_{C}\frac{w^{-1-k} x^{-1+m} (a x)^w (x+\gamma )^{-n}}{b+x}dxdw\\
=\frac{1}{2\pi i}\int_{C}\int_{0}^{\infty}\frac{w^{-1-k} x^{-1+m} (a x)^w (x+\gamma )^{-n}}{b+x}dxdw
\end{multline}
\subsection{First right-hand side contour integral}
\begin{multline}\label{eq:d2}
-\frac{b^{-1+m} e^{i m \pi } (2 i \pi )^{1+k} (-b+\gamma )^{-n} \Phi \left(e^{2 i m \pi },-k,-\frac{i (i \pi +\log (a)+\log (b))}{2 \pi }\right)}{\Gamma(k+1)}\\
=\frac{1}{2\pi i}\int_{C}a^w b^{-1+m+w} \pi  w^{-1-k} (-b+\gamma )^{-n}
   \csc (\pi  (m+w))dw
\end{multline}
\subsection{Second right-hand side contour integral}
\begin{multline}\label{eq:d3}
-\sum _{j=0}^{n-1}
   \sum _{p=0}^j \sum _{l=0}^p \sum _{h=0}^{p-l}\frac{(-1)^{j-l} e^{i m \pi } m^{h+l-p} (2 i \pi )^{1-h+k} \gamma ^{-1+m} (-b+\gamma )^{-n} \left(\frac{-b+\gamma }{\gamma }\right)^j \binom{p}{l} \binom{-l+p}{h} }{j! (-h+k)!}\\
\times \Phi \left(e^{2 i m \pi
   },h-k,-\frac{i (i \pi +\log (a)+\log (\gamma ))}{2 \pi }\right) S_j^{(p)}\\
   =\frac{1}{2\pi i}\sum _{j=0}^{n-1}
   \sum _{p=0}^j \sum _{l=0}^p \sum _{h=0}^{p-l}\int_{C}\frac{(-1)^{j-l} a^w m^{h+l-p} \pi  w^{-1+h-k} \gamma ^{-1+m+w} (-b+\gamma )^{-n} \left(\frac{-b+\gamma }{\gamma
   }\right)^j \binom{p}{l} \binom{-l+p}{h} }{j!}\\
   \times \csc (\pi  (m+w)) S_j^{(p)}dw\\
   =\frac{1}{2\pi i}\int_{C}\sum _{j=0}^{n-1}
   \sum _{p=0}^j \sum _{l=0}^p \sum _{h=0}^{p-l}\frac{(-1)^{j-l} a^w m^{h+l-p} \pi  w^{-1+h-k} \gamma ^{-1+m+w} (-b+\gamma )^{-n} \left(\frac{-b+\gamma }{\gamma
   }\right)^j \binom{p}{l} \binom{-l+p}{h} }{j!}\\
   \times \csc (\pi  (m+w)) S_j^{(p)}dw\\
\end{multline}
%
\section{Main theorem and evaluations}
In this section we evaluate the main theorem and analyze special cases of known integral forms.
\begin{theorem}
\begin{multline}\label{eq:theorem}
\int_0^{\infty } \frac{x^{m-1} \log ^k(a x)}{(b+x) (x+\gamma )^n} \, dx\\
=b^{m-1} (-1)^{m+1} (2 i \pi )^{k+1}
   (\gamma -b)^{-n} \Phi \left(e^{2 i m \pi },-k,-\frac{i (i \pi +\log (a)+\log (b))}{2 \pi }\right)\\
+\sum _{j=0}^{n-1}
   \sum _{p=0}^j \sum _{l=0}^p \sum _{h=0}^{p-l} \frac{(-1)^{j-l} (-1)^m m^{p-l-h} (2 i \pi )^{1-h+k} \gamma ^{m-1}
   (\gamma -b)^{-n} \left(\frac{\gamma -b}{\gamma }\right)^j \binom{p}{l} }{j!}\\
\times \binom{-l+p}{h} (1-h+k)_h \Phi \left(e^{2 i m
   \pi },h-k,-\frac{i (i \pi +\log (a)+\log (\gamma ))}{2 \pi }\right) S_j^{(p)}
\end{multline}
where $0< Re(m) < 1,0< Im(m) < 1$.
\end{theorem}
\begin{proof}
This formula is achieved by observing the right-hand side of equation (\ref{eq:d1}) is equal to the addition of the right-hand sides of equations (\ref{eq:d2}) and (\ref{eq:d3}) relative to equation (\ref{eq2}). The Pochhammer symbol is applied to simplify the right-hand side.
\end{proof}
\begin{example}
Table (2.6.6) in \cite{prud1} can be derived as special cases of the main theorem. For example (2.6.6.17) in \cite{prud1}. This is also a general solution for Problems 3.1  (c) (i) and (i) and Eq. (4.66) in \cite{gordon}.
\begin{multline}
\int_0^{\infty } \frac{\log ^k(x)}{(x+\beta ) (x+\gamma )} \, dx=\frac{(2 i \pi )^{1+k} \left(-\zeta
   \left(-k,\frac{\pi -i \log (\beta )}{2 \pi }\right)+\zeta \left(-k,\frac{\pi -i \log (\gamma )}{2 \pi
   }\right)\right)}{\beta -\gamma }
\end{multline}
where $Re(k)>-1$.
\end{example}
\begin{example}
An alternate form to Eq. (6.4.12) in \cite{erdt1}.
\begin{multline}
\int_0^{\infty } \frac{\log ^n(x)}{x^2+\alpha ^2+2 a x \cos (\theta )} \, dx\\
=\frac{2^{\frac{1}{2}+n} (i \pi
   )^{1+n} }{\sqrt{a^2-2 \alpha
   ^2+a^2 \cos (2 \theta )}}\left(\zeta \left(-n,\frac{\pi -i \log \left(a \cos (\theta )-\frac{\sqrt{a^2-2 \alpha ^2+a^2 \cos (2\theta )}}{\sqrt{2}}\right)}{2 \pi }\right)\right. \\ \left.
-\zeta \left(-n,\frac{\pi -i \log \left(a \cos (\theta
   )+\frac{\sqrt{a^2-2 \alpha ^2+a^2 \cos (2 \theta )}}{\sqrt{2}}\right)}{2 \pi }\right)\right)\\
=-\pi  \cos (\theta ) \frac{\partial ^n\left(\alpha ^{s-2} \csc (s \pi ) \sin ((s-1)
   \theta )\right)}{\partial s^n}
\end{multline}
where $0< Re(s) < 2$.
\end{example}
\begin{example}
\begin{equation}
\int_0^{\infty } \frac{\log (\log (x))}{(x+\beta ) (x+\gamma )} \, dx=\frac{\log (2 \pi ) \log
   \left(\frac{\beta }{\gamma }\right)+2 i \pi  \log \left(\frac{\sqrt[4]{\beta } \Gamma \left(\frac{\pi -i \log
   (\beta )}{2 \pi }\right)}{\sqrt[4]{\gamma } \Gamma \left(\frac{\pi -i \log (\gamma )}{2 \pi
   }\right)}\right)}{\beta -\gamma }
\end{equation}
\end{example}
\begin{example}
Errata 148(1) in \cite{bdh}. 
\begin{equation}
\int_0^{\infty } \frac{\log (\log (x))}{1+x^2} \, dx=\frac{1}{2} \pi  \log \left(\frac{2 i \pi  \Gamma
   \left(\frac{3}{4}\right)^2}{\Gamma \left(\frac{1}{4}\right)^2}\right)
\end{equation}
\end{example}
\begin{example}
Errata 148(2) in \cite{bdh}. 
\begin{equation}
\int_0^{\infty } \frac{\log (\log (x))}{1+x+x^2} \, dx=\frac{\pi  }{\sqrt{3}}\log \left(\frac{4 \sqrt[3]{-1} \pi
   ^{5/3}}{\Gamma \left(\frac{1}{6}\right)^2}\right)
\end{equation}
\end{example}
\begin{example}
Extended Table 148(5) Errata in \cite{bdh}.
\begin{equation}
\int_0^{\infty } \frac{\log (\log (x))}{1-x+x^2} \, dx=\frac{2 \pi  }{3 \sqrt{3}}\left(i \pi +\log \left(\frac{4 \pi ^2 \Gamma
   \left(\frac{5}{6}\right)^3}{\Gamma \left(\frac{1}{6}\right)^3}\right)\right)
\end{equation}
\end{example}
\begin{example}
Table 148(5) Errata in \cite{bdh}.
\begin{equation}
\int_1^{\infty } \frac{\log (\log (x))}{1-x+x^2} \, dx=\frac{\pi }{3 \sqrt{3}} \log \left(\frac{4 \pi ^2 \Gamma
   \left(\frac{5}{6}\right)^3}{\Gamma \left(\frac{1}{6}\right)^3}\right)
   \end{equation}
\end{example}
\begin{example}
\begin{equation}
\int_0^{\infty } \frac{\log (\log (x))}{x^2+\beta ^2} \, dx=\frac{\pi  }{\beta }\log \left(\frac{(1+i) \sqrt{\pi }
   \Gamma \left(\frac{\pi -i \log (i \beta )}{2 \pi }\right)}{\Gamma \left(\frac{\pi -i \log (-i \beta )}{2 \pi
   }\right)}\right)
\end{equation}
\end{example}
\begin{example}
\begin{multline}
\int_0^{\infty } \frac{\sqrt{x} \log ^k(-x)}{\left(1+x^2\right) (1+x)^n} \, dx\\
=i^k 2^{2 k-n} \pi ^{1+k}
   \left(\sqrt[4]{-1} (1-i)^n \zeta \left(-k,\frac{3}{8}\right)-\sqrt[4]{-1} (1-i)^n \zeta
   \left(-k,\frac{7}{8}\right)\right. \\ \left.-(-1)^{3/4} (1+i)^n \left(\zeta \left(-k,\frac{5}{8}\right)-\zeta
   \left(-k,\frac{9}{8}\right)\right)\right)\\
+\sum _{j=0}^{n-1} \sum _{p=0}^j \sum _{l=0}^p \sum _{h=0}^{p-l}
   \frac{(-1)^{j-l} i^{-h+k} 2^{-h+k+l-n-p} \left((1-i)^n (1+i)^j+(1-i)^j (1+i)^n\right) }{j!}\\
\times \left(-2^h+2^{1+k}\right)
   \pi ^{1-h+k} \binom{p}{l} \binom{-l+p}{h} (1-h+k)_h S_j^{(p)} \zeta (h-k)
\end{multline}
\end{example}
\begin{example}
\begin{multline}
\int_0^{\infty } \frac{\left(x^{s-1}-x^{m-1}\right) (x+\gamma )^{-n}}{(b+x) \log (x)} \, dx\\
=\frac{(-b+\gamma
   )^{-n} \left((-1)^m b^m \Phi \left(e^{2 i m \pi },1,-\frac{i (i \pi +\log (b))}{2 \pi }\right)-(-1)^s b^s \Phi
   \left(e^{2 i \pi  s},1,-\frac{i (i \pi +\log (b))}{2 \pi }\right)\right)}{b}\\+\sum _{j=0}^{n-1} \sum _{p=0}^j \sum _{l=0}^p \sum _{h=0}^{p-l} \frac{i (-1)^{j-l} i^{-1-h} m^{-h-l} (2 \pi )^{-h} s^{-h-l} \left(1-\frac{b}{\gamma }\right)^j (-b+\gamma )^{-n} \binom{p}{l} \binom{-l+p}{h} }{\gamma  j!}\\
\times \left(-(-1)^m m^p s^{h+l} \gamma ^m \Phi \left(e^{2 i m\pi },1+h,-\frac{i (i \pi +\log (\gamma ))}{2 \pi }\right)\right. \\ \left.
+(-1)^s m^{h+l} s^p \gamma ^s \Phi \left(e^{2 i \pi s},1+h,-\frac{i (i \pi +\log (\gamma ))}{2 \pi }\right)\right) (-h)_h S_j^{(p)}
\end{multline}
\end{example}
\begin{example}
\begin{multline}
\int_0^{\infty } \frac{(x+\gamma )^{-n} \log ^2\left(-\frac{x}{b}\right) \log \left(\log \left(-\frac{x}{b}\right)\right)}{\sqrt{x} (b+x)} \, dx
=\frac{14 \pi  (-b+\gamma )^{-n} \zeta
   (3)}{\sqrt{b}}\\
-\sum _{j=0}^{n-1} \sum _{p=0}^j \sum _{l=0}^p \sum _{h=0}^{p-l} \frac{(-1)^{j-l} i^{-h} 2^{3-h+l-p} \pi ^{3-h} \left(1-\frac{b}{\gamma }\right)^j (-b+\gamma )^{-n} \binom{p}{l}
   \binom{-l+p}{h} (3-h)_h S_j^{(p)} }{\sqrt{\gamma }
   j!}\\
\times \left(2 \left(\zeta \left(-2+h,\frac{4 \pi +i \log (b)-i \log (\gamma )}{4 \pi }\right)-\zeta \left(-2+h,-\frac{i (2 i \pi -\log (b)+\log (\gamma ))}{4 \pi
   }\right)\right)\right. \\ \left.
 \left(3+i \pi -2 H_{2-h}+\log (4)+2 \log (\pi )\right)+2^h \Phi'\left(-1,-2+h,-\frac{i (2 i \pi -\log (b)+\log (\gamma ))}{2 \pi }\right)\right)
\end{multline}
\end{example}
\begin{example}
\begin{multline}
\int_0^{\infty } \frac{\log ^{1+2 k}(x)}{1-b x+x^2} \, dx\\
=\frac{i^{2 k} (2 \pi )^{2 (1+k)} }{\sqrt{-4+b^2}}\left(-\zeta
   \left(-1-2 k,\frac{\pi -i \log \left(\frac{1}{2} \left(-b-\sqrt{-4+b^2}\right)\right)}{2 \pi }\right)\right. \\ \left.
   +\zeta
   \left(-1-2 k,\frac{\pi -i \log \left(\frac{1}{2} \left(-b+\sqrt{-4+b^2}\right)\right)}{2 \pi
   }\right)\right)
\end{multline}
\end{example}
\begin{example}
\begin{equation}
\int_0^{\infty } \frac{\sqrt{x} \log (\log (x))}{(1+x)^2} \, dx=\pi  \log
   \left(\frac{\left(\frac{1}{3}+\frac{i}{3}\right) e^{-\frac{i}{2}} \sqrt{2 \pi } \Gamma
   \left(-\frac{1}{4}\right)}{\Gamma \left(-\frac{3}{4}\right)}\right)
\end{equation}
\end{example}
\begin{example}
\begin{multline}
\int_0^{\infty } \frac{\log ^k(a x)}{(x+\gamma )^n} \, dx\\
=\sum _{j=0}^{n-1} \sum _{p=0}^j \sum _{l=0}^p \sum
   _{h=0}^{p-l} \frac{(-1)^{j-l} 2^{1-2 h+k-l+p} (i \pi )^{1-h+k} \gamma ^{1-n} \binom{p}{l} \binom{-l+p}{h} }{j!}\\
\times \zeta
   \left(h-k,-\frac{i (i \pi +\log (a)+\log (\gamma ))}{2 \pi }\right) (1-h+k)_h S_j^{(p)}
\end{multline}
\end{example}
\begin{example}
\begin{multline}
\int_0^{\infty } (x+\gamma )^{-n} \log ^k\left(\frac{x}{\gamma }\right) \, dx,\sum _{j=0}^{n-1} \sum _{p=0}^j
   \sum _{l=0}^p \sum _{h=0}^{p-l} \frac{(-1)^{j-l} 2^{1-2 h+k-l+p} \left(-1+2^{h-k}\right) (i \pi )^{1-h+k} }{j!}\\
\times \gamma
   ^{1-n} \binom{p}{l} \binom{-l+p}{h} (1-h+k)_h S_j^{(p)} \zeta (h-k)
\end{multline}
\end{example}
\begin{example}
\begin{multline}
\int_0^{\infty } \frac{(1+x)^{-n} \log (\log (x))}{\sqrt{\log (x)}} \, dx\\
=-\sum _{j=0}^{n-1} \sum _{p=0}^j \sum
   _{l=0}^p \sum _{h=0}^{p-l} \frac{(-1)^{\frac{1}{4}+j-l} i^{-h} 2^{-2 h-l+p} \pi ^{\frac{1}{2}-h} \binom{p}{l}
   \binom{-l+p}{h} \left(\frac{1}{2}-h\right)_h S_j^{(p)} }{j!}\\
\times \left(\left(\sqrt{2} \log (2)+\left(\sqrt{2}-2^{1+h}\right)
   \left(\log (i \pi )+\psi ^{(0)}\left(\frac{1}{2}\right)\right)\right.\right. \\ \left.\left.
+\left(-\sqrt{2}+2^{1+h}\right) \psi
   ^{(0)}\left(\frac{1}{2}-h\right)\right) \zeta \left(\frac{1}{2}+h\right)+\left(-\sqrt{2}+2^{1+h}\right) \zeta
   '\left(\frac{1}{2}+h\right)\right)
\end{multline}
\end{example}
\begin{example}
\begin{multline}
\int_0^{\infty } \frac{\log (\log (x))}{(1+x)^2 \sqrt{\log (x)}} \, dx\\
=\frac{(-1)^{3/4} \left(\left(8-2
   \sqrt{2}+\frac{1}{2} i \left(-4+\sqrt{2}\right) \pi -4 \log (\pi )+\sqrt{2} \log (2 \pi )\right) \zeta
   \left(\frac{3}{2}\right)-\left(-4+\sqrt{2}\right) \zeta '\left(\frac{3}{2}\right)\right)}{4 \sqrt{\pi }}
\end{multline}
\end{example}
\begin{example}
\begin{multline}
\int_0^{\infty } \frac{\left(-x^m+x^s\right) (x+\gamma )^{-n}}{(b+x) (a+\log (x))} \, dx\\
=(-b+\gamma )^{-n} \left(-(-1)^m b^m \Phi \left(e^{2 i m \pi },1,-\frac{i (i \pi +a+\log (b))}{2 \pi
   }\right)\right. \\ \left.
+(-1)^s b^s \Phi \left(e^{2 i \pi  s},1,-\frac{i (i \pi +a+\log (b))}{2 \pi }\right)\right)\\
-\sum _{j=0}^{n-1} \sum _{p=0}^j \sum _{l=0}^p \sum _{h=0}^{p-l} \frac{(-b+\gamma )^{-n} (-1)^{j-l}
   (1+m)^{-h-l} (2 i \pi )^{-h} (1+s)^{-h-l} \left(1-\frac{b}{\gamma }\right)^j \binom{p}{l} \binom{-l+p}{h} }{j!}\\
\times \left(-(-1)^m (1+m)^p (1+s)^{h+l} \gamma ^m \Phi \left(e^{2 i m \pi },1+h,-\frac{i (i \pi +a+\log
   (\gamma ))}{2 \pi }\right)\right. \\ \left.
+(-1)^s (1+m)^{h+l} (1+s)^p \gamma ^s \Phi \left(e^{2 i \pi  s},1+h,-\frac{i (i \pi +a+\log (\gamma ))}{2 \pi }\right)\right) (-h)_h S_j^{(p)}
\end{multline}
\end{example}
\begin{example}
\begin{multline}
\int_0^1 \frac{(x-1) \log (\log (x))}{\sqrt{x} (1+x) \left(1+x^2\right)} \, dx\\
=\frac{1}{4} \pi  \left(2
   \left(\sqrt{2}-2\right) \log (i \pi )+\sqrt{2} \log (16)+4 \sqrt[4]{-1} \log \left(\frac{3 \Gamma
   \left(-\frac{3}{8}\right)}{7 \Gamma \left(-\frac{7}{8}\right)}\right)+8 \log \left(\frac{3 \Gamma
   \left(-\frac{3}{4}\right)}{2 \Gamma \left(-\frac{1}{4}\right)}\right)\right. \\ \left.
+4 (-1)^{3/4} \log \left(\frac{5 \Gamma
   \left(-\frac{5}{8}\right)}{\Gamma \left(-\frac{1}{8}\right)}\right)\right)
\end{multline}
\end{example}
\begin{example}
\begin{multline}
\int_0^{\infty } \frac{\sqrt{x} \log (\log (x))}{(1+x)^4} \, dx\\
=-\frac{C}{2 \pi }+\frac{1}{96} \pi  \left(3 i \pi +2 \left(2 i+\log \left(\frac{64 \pi ^3}{729}\right)-6 \text{log$\Gamma
   $}\left(-\frac{3}{4}\right)+6 \text{log$\Gamma $}\left(-\frac{1}{4}\right)\right)\right)
\end{multline}
\end{example}
\begin{example}
\begin{multline}
\int_0^{\infty } \frac{\sqrt{x} \log (\log (x))}{(1+x)^5} \, dx\\
=\frac{1}{3072 \pi ^3}\left(-896 C \pi ^2+60 i \pi ^5-8 \pi ^4 \left(-16 i+30 \log (3)-15 \log (4)-15 \log (\pi )\right.\right. \\ \left.\left.
+30 \text{log$\Gamma
   $}\left(-\frac{3}{4}\right)-30 \text{log$\Gamma $}\left(-\frac{1}{4}\right)\right)+\psi ^{(3)}\left(\frac{1}{4}\right)-\psi ^{(3)}\left(\frac{3}{4}\right)\right)
   \end{multline}
\end{example}
\begin{example}
\begin{equation}
\int_0^{\infty } \frac{\sqrt{x} \log (\log (x))}{(1+x)^2} \, dx=\pi  \log \left(\frac{2 \sqrt[4]{-1} e^{-\frac{i}{2}} \sqrt{\pi } \Gamma \left(-\frac{1}{4}\right)}{3 \Gamma
   \left(-\frac{3}{4}\right)}\right)
\end{equation}
\end{example}
\begin{example}
\begin{equation}
\int_0^{\frac{\pi }{2}} \log (\log (\tan (t))) \, dt=\pi  \log \left(\frac{(1+i) \sqrt{\pi }
   \Gamma \left(-\frac{1}{4}\right)}{3 \Gamma \left(-\frac{3}{4}\right)}\right)
\end{equation}
\end{example}
\begin{example}
\begin{equation}
\int_0^{\infty } \frac{\sqrt{x} \log (\log (x))}{(1+x)^3} \, dx=-\frac{C}{\pi }+\frac{1}{4} \pi  \log \left(\frac{2 e^{\frac{i \pi }{4}} \sqrt{\pi } \Gamma \left(-\frac{1}{4}\right)}{3 \Gamma
   \left(-\frac{3}{4}\right)}\right)
   \end{equation}
\end{example}
\begin{example}
\begin{multline}
\int_0^{\infty } \frac{\left(-x^m+x^s\right) (x+\gamma )^{-n} \log ^k(a x)}{(b+x) (c+x)} \, dx\\
=\frac{(2 i \pi )^{1+k} (-b+\gamma )^{-n} (-c+\gamma )^{-n} }{b-c}\\
+\sum _{j=0}^{n-1} \sum _{p=0}^j \sum _{l=0}^p \sum _{h=0}^{p-l} \frac{i (-1)^{j-l} i^{-h+k} (1+m)^{-h-l} (2 \pi )^{1-h+k} (1+s)^{-h-l} (-b+\gamma )^{-n} (-c+\gamma )^{-n}
  }{(b-c) j!}\\
\times \left((-c+\gamma )^n \left((-1)^m b^m \Phi
   \left(e^{2 i m \pi },-k,-\frac{i (i \pi +\log (a)+\log (b))}{2 \pi }\right)\right.\right. \\ \left.\left.
+(-1)^{1+s} b^s \Phi \left(e^{2 i \pi  s},-k,-\frac{i (i \pi +\log (a)+\log (b))}{2 \pi }\right)\right)\right. \\ \left.
-(-b+\gamma )^n
   \left((-1)^m c^m \Phi \left(e^{2 i m \pi },-k,-\frac{i (i \pi +\log (a)+\log (c))}{2 \pi }\right)\right.\right. \\ \left.\left.
+(-1)^{1+s} c^s \Phi \left(e^{2 i \pi  s},-k,-\frac{i (i \pi +\log (a)+\log (c))}{2 \pi
   }\right)\right)\right)\\
\times  \left(\left(1-\frac{c}{\gamma }\right)^j (-b+\gamma )^n-\left(1-\frac{b}{\gamma }\right)^j (-c+\gamma )^n\right) \binom{p}{l} \binom{-l+p}{h}\\
 \left((-1)^m (1+m)^p (1+s)^{h+l} \gamma ^m \Phi \left(e^{2 i
   m \pi },h-k,-\frac{i (i \pi +\log (a)+\log (\gamma ))}{2 \pi }\right)\right. \\ \left.
+(-1)^{1+s} (1+m)^{h+l} (1+s)^p \gamma ^s \Phi \left(e^{2 i \pi  s},h-k,-\frac{i (i \pi +\log (a)+\log (\gamma ))}{2 \pi
   }\right)\right) (1-h+k)_h S_j^{(p)}
\end{multline}
\end{example}
\begin{example}
\begin{multline}
\int_0^{\infty } \frac{(1-x) \log (\log (x))}{\sqrt{x} \left(1+x+x^2\right)} \, dx\\
=-\pi  \left(4 \pi +i \left(\log \left(\frac{1728}{25}\right)+2 \log (\pi )-2 \left(\text{log$\Gamma
   $}\left(-\frac{5}{6}\right)+\text{log$\Gamma $}\left(-\frac{1}{6}\right)\right)\right)\right)
\end{multline}
\end{example}
\begin{example}
\begin{multline}
\int_0^{\infty } \frac{(1-x) \log (\log (x))}{\sqrt{x} (1+x)^2 \left(1+x+x^2\right)} \, dx\\
=\frac{1}{6} \pi  \left(6 i-3 i \left(-2+\sqrt{3}\right) \pi -16 \sqrt{3} \log (2)-(24+3 i) \log (3)+3
   \left(4-3 \sqrt{3}\right) \log (\pi )\right. \\ \left.
   +6 \log \left(\frac{16 \Gamma \left(-\frac{1}{4}\right)^4}{\Gamma \left(-\frac{3}{4}\right)^4}\right)+6 \sqrt{3} \log \left(\frac{5 \Gamma \left(-\frac{5}{6}\right)
   \Gamma \left(\frac{1}{6}\right)}{\Gamma \left(-\frac{1}{6}\right)}\right)\right)
\end{multline}
\end{example}
\begin{example}
\begin{multline}
\int_0^{\infty } \frac{(1-x) \log (\log (x))}{\sqrt{x} (1+x)^2 \left(1-x+x^2\right)} \, dx\\
=\frac{1}{18} \pi  \left(6 i+3 i \pi +2 i \sqrt{3} \cosh ^{-1}(2)+6 \log \left(\frac{4 \pi  \Gamma
   \left(-\frac{1}{4}\right)^2}{3^{3/2} \Gamma \left(-\frac{3}{4}\right)^2}\right)\right)
\end{multline}
\end{example}
\begin{example}
\begin{multline}
\int_0^{\infty } \frac{(1-x) \log (\log (x))}{\sqrt{x} (1+x)^2} \, dx= 2\int_0^{\infty } \frac{\left(1-x^2\right) \log (\log (x))}{\left(1+x^2\right)^2}\, dx=i \pi
\end{multline}
\end{example}
\subsection{Extended Kolbig forms}
In this section we look at extended forms by the some definite integral work published by K\"{o}lbig \cite{kolbig}. Similar work was published in Adamchik \cite{victor}, Geddes \cite{geddes} and Gr\"{o}bner \cite{grobner}.
\begin{example}
Eq. (3.5) and (3.7) in \cite{kolbig}.
\begin{multline}
\int_0^{\infty } x^m (x+\gamma )^{-n} \log ^k(x) \, dx\\
=\sum _{j=0}^{n-1} \sum _{p=0}^j \sum _{l=0}^p \sum
   _{h=0}^{p-l} \frac{(-1)^{2+j-l+m} (2+m)^{-h-l+p} (2 i \pi )^{1-h+k} \gamma ^{1+m-n} \binom{p}{l} \binom{-l+p}{h}
  }{j!}\\
\times  \Phi \left(e^{2 i (2+m) \pi },h-k,-\frac{i (i \pi +\log (\gamma ))}{2 \pi }\right) (1-h+k)_h
   S_j^{(p)}
\end{multline}
\end{example}
\begin{example}
Eq. (3.8) in \cite{kolbig}.
\begin{multline}
\int_0^{\infty } x^m (x-\gamma )^{-n} \log ^k(x) \, dx\\
=\sum _{j=0}^{n-1} \sum _{p=0}^j \sum _{l=0}^p \sum
   _{h=0}^{p-l} \frac{(-1)^{2+j-l+m} (2+m)^{-h-l+p} (2 i \pi )^{1-h+k} (-\gamma )^{1+m-n} \binom{p}{l}
   \binom{-l+p}{h} }{j!}\\ \times \Phi \left(e^{2 i (2+m) \pi },h-k,-\frac{i (i \pi +\log (-\gamma ))}{2 \pi }\right) (1-h+k)_h
   S_j^{(p)}
\end{multline}
\end{example}
\begin{example}
Eq. (3.9) in \cite{kolbig}.
\begin{multline}
\int_0^{\infty } x^m (x+\beta )^{\frac{1}{2}-n} \log ^k(x) \, dx\\
=\sum _{j=0}^{\left\lfloor
   n-\frac{3}{2}\right\rfloor } \sum _{p=0}^j \sum _{l=0}^p \sum _{h=0}^{p-l} \frac{(-1)^{2+j-l+m} (2+m)^{-h-l+p} (2
   i \pi )^{1-h+k} \beta ^{\frac{3}{2}+m-n} \binom{p}{l} \binom{-l+p}{h}}{j!}\\
\times  \Phi \left(e^{2 i (2+m) \pi },h-k,-\frac{i
   (i \pi +\log (\beta ))}{2 \pi }\right) (1-h+k)_h S_j^{(p)}
\end{multline}
where $n=\frac{2p+1}{2}$ and $p$ is any natural number.
\end{example}
\begin{example}
\begin{multline}
\int_0^{\infty } \left(-x^m+x^s\right)  \frac{\log ^k(a x)}{(x+\gamma )^{n}} \, dx\\
=\sum _{j=0}^{n-1} \sum _{p=0}^j \sum _{l=0}^p \sum _{h=0}^{p-l} \frac{i (-1)^{j-l} i^{-h+k} (2+m)^{-h-l} (2 \pi )^{1-h+k}
   (2+s)^{-h-l} \gamma ^{1-n} \binom{p}{l} \binom{-l+p}{h} }{j!}\\
\times \left((-1)^{1+m} (2+m)^p (2+s)^{h+l} \gamma ^m \Phi \left(e^{2 i m \pi },h-k,-\frac{i (i \pi +\log (a)+\log (\gamma ))}{2 \pi }\right)\right. \\ \left.
+(-1)^s
   (2+m)^{h+l} (2+s)^p \gamma ^s \Phi \left(e^{2 i \pi  s},h-k,-\frac{i (i \pi +\log (a)+\log (\gamma ))}{2 \pi }\right)\right) (1-h+k)_h S_j^{(p)}
\end{multline}
\end{example}
\begin{example}
\begin{equation}
\int_0^{\infty } \frac{\log (x) \log (\log (x))}{(1-x) \sqrt{x}} \, dx=\pi ^2 \log \left(\frac{A^{12}}{4 \sqrt[3]{2} e \pi  i}\right)
\end{equation}
\end{example}
\begin{example}
\begin{multline}
\int_0^{\infty } \frac{\sqrt{\log (x)} \log (\log (x))}{(1-x) \sqrt{x}} \, dx\\
=(1+i) \pi ^{3/2} \left(\left(-1+2 \sqrt{2}\right) (\pi -2 i \log (2 \pi )) \zeta \left(-\frac{1}{2}\right)+2 i
   Li_{-\frac{1}{2}}\left(-1\right)\right)
\end{multline}
\end{example}
\section{Definite integrals involving Cauchy principal value}
In this section we look at special cases using equation (\ref{eq:theorem}) where the definite integral has at least one pole. The treatment of such integrals is detailed in Arfken et al. \cite{arfken} pp. 457-460. and Chapter 4 pp. 102 in Gordon \cite{gordon}. In some cases l;Hopital's rule was applied to the right-hand side in order to simplify the closed form solution.
\begin{example}
\begin{multline}
\int_0^{\infty } \frac{\left(x^m-x^s\right) \log ^k(a x)}{b^2-x^2} \, dx\\
=-\frac{2^k (i \pi )^{1+k} }{b}\left(\left((-1)^m (-b)^m \Phi \left(e^{2 i m \pi },-k,-\frac{i (i \pi +\log (a)+\log
   (-b))}{2 \pi }\right)\right.\right. \\ \left.\left.
-(-1)^m b^m \Phi \left(e^{2 i m \pi },-k,-\frac{i (i \pi +\log (a)+\log (b))}{2 \pi }\right)\right.\right. \\ \left.\left.
+(-1)^s \left(-(-b)^s \Phi \left(e^{2 i \pi  s},-k,-\frac{i (i \pi +\log
   (a)+\log (-b))}{2 \pi }\right)\right.\right.\right. \\ \left.\left.\left.
+b^s \Phi \left(e^{2 i \pi  s},-k,-\frac{i (i \pi +\log (a)+\log (b))}{2 \pi }\right)\right)\right) \right)
\end{multline}
where $-1< Re(m)<1,-1< Re(s)<1,Im(b)\geq 0$.
\end{example}
\begin{example}
\begin{multline}
\int_0^{\infty } \frac{\log (\log (x))}{\sqrt{x} (x+1) \log (x)} \, dx=\frac{1}{4} \left(\pi ^2-2 i \pi  \log (4 \pi )-2 i \left(\gamma _1\left(\frac{1}{4}\right)-\gamma
   _1\left(\frac{3}{4}\right)\right)\right)
\end{multline}
\end{example}
\begin{example}
\begin{multline}
\int_0^{\infty } \frac{x^m-x^s}{\left(-b^2+x^2\right) \left(a^2-\log ^2(x)\right)} \, dx\\
=\frac{(-1)^{1+m} (-b)^m}{4 a b} \left(\Phi \left(e^{2 i m \pi },1,\frac{i a+\pi -i \log (-b)}{2 \pi
   }\right)\right. \\ \left.
-\Phi \left(e^{2 i m \pi },1,-\frac{i (a+i \pi +\log (-b))}{2 \pi }\right)\right)+(-1)^m b^m \left(\Phi \left(e^{2 i m \pi },1,\frac{i a+\pi -i \log (b)}{2 \pi }\right)\right. \\ \left.
-\Phi
   \left(e^{2 i m \pi },1,-\frac{i (a+i \pi +\log (b))}{2 \pi }\right)\right)+(-1)^s (-b)^s \left(\Phi \left(e^{2 i \pi  s},1,\frac{i a+\pi -i \log (-b)}{2 \pi }\right)\right. \\ \left.
-\Phi \left(e^{2 i \pi
    s},1,-\frac{i (a+i \pi +\log (-b))}{2 \pi }\right)\right)+(-1)^{1+s} b^s \left(\Phi \left(e^{2 i \pi  s},1,\frac{i a+\pi -i \log (b)}{2 \pi }\right)\right. \\ \left.
-\Phi \left(e^{2 i \pi  s},1,-\frac{i
   (a+i \pi +\log (b))}{2 \pi }\right)\right)
\end{multline}
where $Im(b)< 0,Re(b)>0$.
\end{example}
\begin{example}
\begin{multline}
\int_0^{\infty } \frac{\left(x^m-x^s\right) (x+\gamma )^{-n} \log ^k(a x)}{b^2-x^2} \, dx\\
=-\frac{2^k (i \pi )^{1+k}}{b} \left((b+\gamma )^{-n} \left((-1)^m (-b)^m \Phi \left(e^{2 i m \pi
   },-k,-\frac{i (i \pi +\log (a)+\log (-b))}{2 \pi }\right)\right.\right. \\ \left.\left.
-(-1)^s (-b)^s \Phi \left(e^{2 i \pi  s},-k,-\frac{i (i \pi +\log (a)+\log (-b))}{2 \pi }\right)\right)\right. \\ \left.
+(-b+\gamma )^{-n}
   \left(-(-1)^m b^m \Phi \left(e^{2 i m \pi },-k,-\frac{i (i \pi +\log (a)+\log (b))}{2 \pi }\right)\right.\right. \\ \left.\left.
+(-1)^s b^s \Phi \left(e^{2 i \pi  s},-k,-\frac{i (i \pi +\log (a)+\log (b))}{2 \pi
   }\right)\right)\right)\\
+\sum _{j=0}^{n-1} \sum _{p=0}^j \sum _{l=0}^p \sum _{h=0}^{p-l} \frac{(-1)^{1+j-l} 2^{-h+k} (1+m)^{-h-l} (i \pi )^{1-h+k} (1+s)^{-h-l} (-b+\gamma )^{-n}
   (b+\gamma )^{-n} }{b j!}\\ \times \left(-\left(1-\frac{b}{\gamma }\right)^j (b+\gamma )^n+(-b+\gamma )^n \left(\frac{b+\gamma }{\gamma }\right)^j\right) \binom{p}{l} \binom{-l+p}{h} \\
\left(-(-1)^m (1+m)^p
   (1+s)^{h+l} \gamma ^m \Phi \left(e^{2 i m \pi },h-k,-\frac{i (i \pi +\log (a)+\log (\gamma ))}{2 \pi }\right)\right. \\ \left.
+(-1)^s (1+m)^{h+l} (1+s)^p \gamma ^s \Phi \left(e^{2 i \pi  s},h-k,-\frac{i
   (i \pi +\log (a)+\log (\gamma ))}{2 \pi }\right)\right) (1-h+k)_h S_j^{(p)}
\end{multline}
where $Im(b)\leq 0$.
\end{example}
\begin{example}
\begin{multline}
\int_0^{\infty } \frac{(1-x) \log (\log (x))}{\sqrt{x} \left(-b^2+x^2\right) (x+\gamma )} \, dx\\
=\left(-\frac{1}{4}+\frac{i \log (b)}{4 \pi }\right) \left(-\frac{i (1+b) \pi ^2}{2
   b^{3/2} (b-\gamma )}-\frac{(1+b) \pi  (\log (4)+\log (\pi ))}{b^{3/2} (b-\gamma )}\right)\\
+\left(\frac{1}{2}-\frac{\pi -i \log (b)}{4 \pi }\right) \left(\frac{i (1+b) \pi ^2}{2 b^{3/2}
   (b-\gamma )}+\frac{(1+b) \pi  (\log (4)+\log (\pi ))}{b^{3/2} (b-\gamma )}\right)\\
+\left(\frac{1}{2}-\frac{\pi -i \log (-b)}{4 \pi }\right) \left(-\frac{i (-1+b) \pi ^2}{2 \sqrt{-b} b
   (b+\gamma )}+\frac{(1-b) \pi  (\log (4)+\log (\pi ))}{\sqrt{-b} b (b+\gamma )}\right)\\
+\left(-\frac{1}{4}+\frac{i \log (-b)}{4 \pi }\right) \left(\frac{i (-1+b) \pi ^2}{2 \sqrt{-b} b
   (b+\gamma )}+\frac{(-1+b) \pi  (\log (4)+\log (\pi ))}{\sqrt{-b} b (b+\gamma )}\right)\\
+\left(-\frac{i \pi ^2 (1+\gamma )}{(b-\gamma ) \sqrt{\gamma } (b+\gamma )}-\frac{2 \pi  (1+\gamma )
   (\log (4)+\log (\pi ))}{(b-\gamma ) \sqrt{\gamma } (b+\gamma )}\right) \left(\frac{1}{2}-\frac{\pi -i \log (\gamma )}{4 \pi }\right)\\
+\left(\frac{i \pi ^2 (1+\gamma )}{(b-\gamma )
   \sqrt{\gamma } (b+\gamma )}+\frac{2 \pi  (1+\gamma ) (\log (4)+\log (\pi ))}{(b-\gamma ) \sqrt{\gamma } (b+\gamma )}\right) \left(-\frac{1}{4}+\frac{i \log (\gamma )}{4 \pi
   }\right)\\
+\frac{(-1+b) \pi  \left(\frac{1}{2} (-\log (2)-\log (\pi ))+\log \left(-1+\frac{\pi -i \log (-b)}{4 \pi }\right)+\text{log$\Gamma $}\left(-1+\frac{\pi -i \log (-b)}{4 \pi
   }\right)\right)}{\sqrt{-b} b (b+\gamma )}\\
+\frac{(1-b) \pi  \left(\frac{1}{2} (-\log (2)-\log (\pi ))+\log \left(-\frac{1}{4}-\frac{i \log (-b)}{4 \pi }\right)+\text{log$\Gamma
   $}\left(-\frac{1}{4}-\frac{i \log (-b)}{4 \pi }\right)\right)}{\sqrt{-b} b (b+\gamma )}\\
-\frac{(1+b) \pi  \left(\frac{1}{2} (-\log (2)-\log (\pi ))+\log \left(-1+\frac{\pi -i \log (b)}{4
   \pi }\right)+\text{log$\Gamma $}\left(-1+\frac{\pi -i \log (b)}{4 \pi }\right)\right)}{b^{3/2} (b-\gamma )}\\
+\frac{(1+b) \pi  \left(\frac{1}{2} (-\log (2)-\log (\pi ))+\log
   \left(-\frac{1}{4}-\frac{i \log (b)}{4 \pi }\right)+\text{log$\Gamma $}\left(-\frac{1}{4}-\frac{i \log (b)}{4 \pi }\right)\right)}{b^{3/2} (b-\gamma )}\\
+\frac{2 \pi  (1+\gamma )
   \left(\frac{1}{2} (-\log (2)-\log (\pi ))+\log \left(-1+\frac{\pi -i \log (\gamma )}{4 \pi }\right)+\text{log$\Gamma $}\left(-1+\frac{\pi -i \log (\gamma )}{4 \pi
   }\right)\right)}{(b-\gamma ) \sqrt{\gamma } (b+\gamma )}\\
-\frac{2 \pi  (1+\gamma ) \left(\frac{1}{2} (-\log (2)-\log (\pi ))+\log \left(-\frac{1}{4}-\frac{i \log (\gamma )}{4 \pi
   }\right)+\text{log$\Gamma $}\left(-\frac{1}{4}-\frac{i \log (\gamma )}{4 \pi }\right)\right)}{(b-\gamma ) \sqrt{\gamma } (b+\gamma )}
\end{multline}
where $Im(b)\leq 0$.
\end{example}
\begin{example}
\begin{multline}
\int_0^{\infty } \frac{\log (\log (x))}{\sqrt{x} \left(1-x^2\right)} \, dx=\frac{1}{4} \pi  \left((-1+i) \pi +(4+4 i) \log (2)+2 \log \left(\frac{\pi  \Gamma
   \left(-\frac{1}{4}\right)^2}{9 \Gamma \left(-\frac{3}{4}\right)^2}\right)\right)
\end{multline}
\end{example}
\begin{example}
\begin{multline}
\int_0^{\infty } \frac{\left(-1+\sqrt{x}\right) \log (\log (x))}{\sqrt[4]{x} \left(b^2-x^2\right) (x+\gamma )} \, dx\\
=\frac{\left(\frac{1}{4}+\frac{i}{4}\right) \sqrt[4]{-1} \pi }{\sqrt[4]{-b} b^{5/4}
   (b-\gamma ) \sqrt[4]{\gamma } (b+\gamma )}
   \left(\pi  \sqrt[4]{\gamma } \left(\left(1+\sqrt{-b}\right) \sqrt[4]{b} (b-\gamma )+\left(1+\sqrt{b}\right) \sqrt[4]{-b} (b+\gamma )\right)\right. \\ \left.
+4 i \sqrt[4]{-b} b^{5/4} \left(1+\sqrt{\gamma
   }\right) \log (2)+4 i \sqrt[4]{-b} b^{5/4} \left(1+\sqrt{\gamma }\right) \log (i \pi )-2 i \sqrt[4]{\gamma } \left(\left(1+\sqrt{-b}\right) \sqrt[4]{b} (b-\gamma )\right.\right. \\ \left.\left.
+\left(1+\sqrt{b}\right)
   \sqrt[4]{-b} (b+\gamma )\right) \log (2 \pi )+(2+2 i) \sqrt[4]{\gamma } \left(\sqrt[4]{b} (b-\gamma ) \left(\left(i+\sqrt{-b}\right) \log \left(\frac{\Gamma \left(\frac{\pi -i \log
   (-b)}{8 \pi }\right)}{2 \Gamma \left(\frac{5}{8}-\frac{i \log (-b)}{8 \pi }\right)}\right)\right.\right.\right. \\ \left.\left.\left.
+\left(1+i \sqrt{-b}\right) \log \left(\frac{\Gamma \left(\frac{3}{8}-\frac{i \log (-b)}{8 \pi
   }\right)}{2 \Gamma \left(\frac{7}{8}-\frac{i \log (-b)}{8 \pi }\right)}\right)\right)+\sqrt[4]{-b} (b+\gamma ) \left(\left(i+\sqrt{b}\right) \log \left(\frac{\Gamma \left(\frac{\pi -i
   \log (b)}{8 \pi }\right)}{2 \Gamma \left(\frac{5}{8}-\frac{i \log (b)}{8 \pi }\right)}\right)\right.\right.\right. \\ \left.\left.\left.
+\left(1+i \sqrt{b}\right) \log \left(\frac{\Gamma \left(\frac{3}{8}-\frac{i \log (b)}{8 \pi
   }\right)}{2 \Gamma \left(\frac{7}{8}-\frac{i \log (b)}{8 \pi }\right)}\right)\right)\right)\right. \\ \left.
+\frac{(4+4 i) b^{9/4} \left(\left(i+\sqrt{\gamma }\right) \log \left(\frac{\Gamma
   \left(\frac{\pi -i \log (\gamma )}{8 \pi }\right)}{2 \Gamma \left(\frac{5}{8}-\frac{i \log (\gamma )}{8 \pi }\right)}\right)+\left(1+i \sqrt{\gamma }\right) \log \left(\frac{\Gamma
   \left(\frac{3}{8}-\frac{i \log (\gamma )}{8 \pi }\right)}{2 \Gamma \left(\frac{7}{8}-\frac{i \log (\gamma )}{8 \pi }\right)}\right)\right)}{(-b)^{3/4}}\right)
\end{multline}
where $Im(b)\leq 0$.
\end{example}
\begin{example}
\begin{multline}
\int_0^{\infty } \frac{\log (\log (x))}{\left(1+\sqrt{x}\right) \sqrt[4]{x} \left(1-x^2\right)} \, dx\\
=\frac{1}{8} \pi  \left(-2 i+\left(-1+2 i \sqrt{2}\right) \pi -4 \log (\pi )+\log
   \left(\frac{64^{(-2+i)+2 \sqrt{2}} \Gamma \left(\frac{1}{4}\right)^8}{\Gamma \left(\frac{3}{4}\right)^8}\right)\right. \\ \left.
+4 \sqrt{2} \log \left(\frac{\pi  \Gamma \left(\frac{5}{8}\right) \Gamma
   \left(\frac{7}{8}\right)}{\Gamma \left(\frac{1}{8}\right) \Gamma \left(\frac{3}{8}\right)}\right)\right)
\end{multline}
\end{example}
\begin{example}
\begin{multline}
\int_0^{\infty } \frac{\sqrt{x} \log (\log (x))}{(-1+x)^2 (1+x)} \,
   dx=\frac{1}{2} \left(-\log (2)+\pi  \log \left(-\frac{(1+i) \sqrt{\frac{2}{\pi
   }} \Gamma \left(\frac{1}{4}\right)}{\Gamma
   \left(-\frac{1}{4}\right)}\right)\right)
\end{multline}
\end{example}
\begin{example}
\begin{multline}
\int_0^{\infty } \frac{\left(\sqrt{x}-x\right) \log (\log (x))}{(-1+x)^2
   \left(1-x^2\right)} \, dx\\
=\frac{1}{32} \left(-8 \log (2)+\pi  \left(i-(1+2 i)
   \pi -4 \log (\pi )+8 \log \left(-\frac{2 \Gamma
   \left(\frac{1}{4}\right)}{\Gamma
   \left(-\frac{1}{4}\right)}\right)\right)\right)
\end{multline}
\end{example}
\begin{example}
\begin{multline}
\int_0^{\infty } \frac{\left(-1+\sqrt[4]{x}\right) \log (\log (x))}{\sqrt{x} \left(-b^2+x^2\right) (x+\gamma )} \, dx\\
=\frac{1}{4 b^2 (b-\gamma ) \sqrt{\gamma } (b+\gamma )}\left(-4 \sqrt{2} b^2 \pi  \sqrt[4]{\gamma } \left(\frac{i \pi}{2}+\log (2 \pi )\right)\right. \\ \left.
-2 \sqrt{2} (-b)^{3/4} \pi  (b-\gamma ) \sqrt{\gamma } \left(\frac{i \pi }{2}+\log (2 \pi )\right)+\sqrt{-b} (b-\gamma ) \sqrt{\gamma } (\pi -i \log (-b))
   \left(\frac{i \pi }{2}+\log (4 \pi )\right)\right. \\ \left.
+\sqrt{-b} (b-\gamma ) \sqrt{\gamma } (\pi +i \log (-b)) \left(\frac{i \pi }{2}+\log (4 \pi )\right)-\sqrt{b} \sqrt{\gamma } (b+\gamma ) (\pi -i
   \log (b)) \left(\frac{i \pi }{2}+\log (4 \pi )\right)\right. \\ \left.
-\sqrt{b} \sqrt{\gamma } (b+\gamma ) (\pi +i \log (b)) \left(\frac{i \pi }{2}+\log (4 \pi )\right)+\sqrt{2} b^{3/4} \pi  \sqrt{\gamma
   } (b+\gamma ) \left(i \pi +\log \left(4 \pi ^2\right)\right)\right. \\ \left.
+b^2 (i \pi +2 \log (4 \pi )) (\pi -i \log (\gamma ))\right. \\ \left.
+b^2 (i \pi +2 \log (4 \pi )) (\pi +i \log (\gamma ))+4 \sqrt[4]{-1}
   (-b)^{3/4} \pi  (b-\gamma ) \sqrt{\gamma } \left(\log \left(\frac{\Gamma \left(\frac{\pi -i \log (-b)}{8 \pi }\right)}{2 \Gamma \left(\frac{5}{8}-\frac{i \log (-b)}{8 \pi
   }\right)}\right)\right.\right. \\ \left.\left.
-i \log \left(\frac{\Gamma \left(\frac{3}{8}-\frac{i \log (-b)}{8 \pi }\right)}{2 \Gamma \left(\frac{7}{8}-\frac{i \log (-b)}{8 \pi }\right)}\right)\right)-4 \sqrt[4]{-1}
   b^{3/4} \pi  \sqrt{\gamma } (b+\gamma ) \left(\log \left(\frac{\Gamma \left(\frac{\pi -i \log (b)}{8 \pi }\right)}{2 \Gamma \left(\frac{5}{8}-\frac{i \log (b)}{8 \pi }\right)}\right)\right.\right. \\ \left.\left.
-i \log \left(\frac{\Gamma \left(\frac{3}{8}-\frac{i \log (b)}{8 \pi }\right)}{2 \Gamma \left(\frac{7}{8}-\frac{i \log (b)}{8 \pi }\right)}\right)\right)+8 \sqrt[4]{-1} b^2 \pi 
   \sqrt[4]{\gamma } \left(\log \left(\frac{\Gamma \left(\frac{\pi -i \log (\gamma )}{8 \pi }\right)}{2 \Gamma \left(\frac{5}{8}-\frac{i \log (\gamma )}{8 \pi }\right)}\right)-i \log
   \left(\frac{\Gamma \left(\frac{3}{8}-\frac{i \log (\gamma )}{8 \pi }\right)}{2 \Gamma \left(\frac{7}{8}-\frac{i \log (\gamma )}{8 \pi }\right)}\right)\right)\right. \\ \left.
-2 \sqrt{-b} \pi  (b-\gamma )
   \sqrt{\gamma } \left(\log \left(32 \pi ^3\right)-2 \log (-\pi -i \log (-b))-2 \text{log$\Gamma $}\left(-\frac{\pi +i \log (-b)}{4 \pi }\right)\right)\right. \\ \left.
+2 \sqrt{-b} \pi  (b-\gamma )
   \sqrt{\gamma } \left(\log \left(32 \pi ^3\right)-2 \log (-3 \pi -i \log (-b))-2 \text{log$\Gamma $}\left(-\frac{3}{4}-\frac{i \log (-b)}{4 \pi }\right)\right)\right. \\ \left.
+2 \sqrt{b} \pi  \sqrt{\gamma
   } (b+\gamma ) \left(\log \left(32 \pi ^3\right)-2 \log (-\pi -i \log (b))-2 \text{log$\Gamma $}\left(-\frac{\pi +i \log (b)}{4 \pi }\right)\right)\right. \\ \left.
-2 \sqrt{b} \pi  \sqrt{\gamma } (b+\gamma
   ) \left(\log \left(32 \pi ^3\right)-2 \log (-3 \pi -i \log (b))-2 \text{log$\Gamma $}\left(-\frac{3}{4}-\frac{i \log (b)}{4 \pi }\right)\right)\right. \\ \left.
-4 b^2 \pi  \left(\log \left(32 \pi
   ^3\right)-2 \log (-\pi -i \log (\gamma ))-2 \text{log$\Gamma $}\left(-\frac{\pi +i \log (\gamma )}{4 \pi }\right)\right)\right. \\ \left.
+4 b^2 \pi  \left(\log \left(32 \pi ^3\right)-2 \log (-3 \pi -i
   \log (\gamma ))-2 \text{log$\Gamma $}\left(-\frac{3}{4}-\frac{i \log (\gamma )}{4 \pi }\right)\right)\right)
\end{multline}
where $Im(\gamma)\geq 0,Im(b)\leq 0,Re(b)\geq 0$.
\end{example}
\begin{example}
\begin{multline}
\int_0^{\infty } \frac{\left(-1+\sqrt[4]{x}\right) \log (\log (x))}{(-1+x)^2 \sqrt{x} (1+x)} \, dx\\
=\frac{1}{16} \left(\left((1-5 i)+2 i \sqrt{2}\right) \pi ^2+\log (16)+2 \pi 
   \left(-i-(13-i) \log (2)+\sqrt{2} \log (64)\right.\right. \\ \left.\left.
+\left(-5+2 \sqrt{2}\right) \log (\pi )+6 \log \left(\Gamma \left(\frac{1}{4}\right)\right)-6 \log \left(\Gamma
   \left(\frac{3}{4}\right)\right)\right.\right. \\ \left.\left.
+\log \left(81 \left(\frac{\Gamma \left(\frac{5}{8}\right)}{\Gamma \left(\frac{1}{8}\right)}\right)^{4 \sqrt[4]{-1}} \left(\frac{\Gamma
   \left(\frac{3}{8}\right)}{\Gamma \left(\frac{7}{8}\right)}\right)^{4 (-1)^{3/4}}\right)+4 \text{log$\Gamma $}\left(-\frac{3}{4}\right)-4 \text{log$\Gamma
   $}\left(-\frac{1}{4}\right)\right)\right)
\end{multline}
\end{example}
%
%

%
\end{document}